\documentclass{amsart}

\usepackage{xcolor} 
\usepackage{graphicx} 
\usepackage{enumerate}

\usepackage{amsmath}
\usepackage{amssymb}
\usepackage{amsrefs}

\newcommand\intt{\operatorname{int}}
\newcommand\pr{\operatorname{pr}}
\newcommand\cov{\operatorname{cov}}
 
\theoremstyle{plain}
\newtheorem{theorem}{Theorem}[section]
\newtheorem{lemma}[theorem]{Lemma}
\newtheorem{proposition}[theorem]{Proposition}
\newtheorem{corollary}[theorem]{Corollary}
\newtheorem{definition}[theorem]{Definition}

\theoremstyle{definition}

\newtheorem{question}{Question}

\theoremstyle{remark}

\numberwithin{equation}{section}

\title{
Densely $k$-separable compacta are densely separable
}

\date{\today}

\author{Alan Dow}

\author{Istv\'an Juh\'asz}

\begin{document}
\begin{abstract}
A space has $\sigma$-compact tightness if the closures of
$\sigma$-compact  subsets determines the topology. We consider a dense
set variant that we call densely $k$-separable. We consider
 the question of whether every densely $k$-separable  space
 is separable.  The somewhat
surprising answer is that this property, for compact spaces,
 implies that every dense set is separable. The path to this result
 relies on the known connections
 established between  $\pi$-weight and the 
density of all  dense subsets, or more precisely, the
 cardinal invariant $\delta(X)$.
\end{abstract}

\maketitle

\section{Introduction}

In \cite{ArhStav}, Arhangel' ski{\u\i} and Stavrova
 explored generalizations of the usual
notion of tightness. In particular, they introduced the cardinal
invariant $t_k^*(X)$ of a space $X$, which, in the countable case has
been  called 
$\sigma$-compact tightness. It was proven that in the class of compact
spaces this invariant corresponded to the tightness, $t(X)$, of
$X$. They raised the question of whether $\sigma$-compact tightness
implied tightness in all Hausdorff spaces. Some partial results were
obtained in \cite{DowMoore} but the question remains open. This, and
other generalizations of  tightness were considered in
\cite{JuhvM}. The authors of \cite{JuhvM} formulated a related concept,
 that we will call 
densely $k$-separable,
and posed
 (unpublished)   the
 problem of whether every compact space that 
was densely $k$-separable
was, in fact,  separable. We answer this
 question and discover the, unexpectedly strong, answer that such
 spaces will actually have countable $\pi$-weight.
 The definition (\ref{dk*}) is analogous to $t_k^*(X)$ but
we adopt the notation $\delta_k(X)$  
because
 it  also bears similarities to the cardinality
invariant $\delta(X)$ which is the supremum of the densities of dense
subsets, and the proof of our main result uses the main
 result from \cite{JuhSh} that $\delta(X) = \pi(X)$ for compact spaces
 $X$ (see also \cites{Sapirovski,10years}).  
In the next section we establish some basic properties of
 $\delta_k(X)$ that we will need for the proof of the main result. In the
 final section we only consider the case that $\delta_k(X)$ is countable
 and we  prove the main result. Throughout the paper, when we refer to
 a topological space we intend that the space be a regular Hausdorff
 space. A subset $W$ of a space $X$ is regular closed if $W$ is equal
 to the closure of its interior $\intt_X(W)$. We will work with
 regular closed sets rather than their interiors, 
i.e. regular open sets,  since regular closed subsets of compact
spaces are compact.

 For a space $X$, the compact covering number, $\cov_k(X)$, is
 the least cardinal $\tau$ such that $X$ can be covered by 
 $\tau$ many compact sets.
In \cite{ArhStav}, the 
value of $t_k^*(X)$ is the least cardinal $\tau$
satisfying that the closure of every subset $A$ of $X$ is equal to the
union of the closures of subsets $Y$ of $A$ satisfying that
 $\cov_k(Y)\leq\tau$.
Following
that theme we introduce  dense set variations.

 \begin{definition} 
For a space $X$ we set\label{dk*} 
$$d_k(X)=\min\{ \cov_k(Y) :  Y \ \mbox{is a dense subset of }\ X\}$$
and then 
$$\delta_k(X)=\sup\{ d_k(Y) :  Y \ \mbox{is a dense subset of }\ X\}~~.$$
\end{definition}

Then the natural problem  to ask 
is under what conditions on a space $X$ do we have
$              \delta_k(X) =  \delta(X)$?
It is easy to find examples of $\sigma$-compact
non-separable spaces which witness
 $d_k(X) < d(X)$ can hold. 
  We do not know  any example of a space $X$ for which
  $\delta_k(X) < \delta(X)$ and raise it as an open problem.

  \begin{question} Is there a
     regular    Hausdorff  space    $X$ such that
$\delta_k(X) <\delta(X)$?
  \end{question}

It is natural to call  a space $X$ \(k\)-\textit{separable\/}
if  $ d_k(X) \le \aleph_0$   and 
\textit{densely \(k\)-separable\/}  if $\delta_k(X) \le \aleph_0$.
The question, posed by Jan van Mill,
 that motivated this paper was whether every compact space
$X$ with $\delta_k(X)=\aleph_0$ is separable.
We succeed in  establishing the stronger
result that compact densely $k$-separable spaces
have  countable $\pi$-weight. The
$\pi$-weight, $\pi(X)$, of a space $X$ is the minimum cardinality of a
$\pi$-base for $X$. A $\pi$-base for a space $X$ is a family of
non-empty open sets that has the property that every non-empty open
subset of $X$ includes a member of that family.  A local $\pi$-base at
a point $x$ is a family of non-empty open sets with the property that
every neighborhood of the point includes a member of that family.
One of the main
elements of the proof is the connection to $\delta(X)$ mentioned
above. It is shown in \cite{JuhSh} that
for compact Hausdorff space $X$, 
 $\delta(X) = \pi(X)$ and 
we will need the following stronger result  from \cite{JuhSh}.

\begin{proposition}[\cite{JuhSh}]
For each compact space $X$, there is a dense\label{delta}
subset $Y$ of $X$ such that  $d(Y)  = \pi(X)$.
\end{proposition}

A subset $Y$ of a space $X$ is $G_\delta$-dense if $Y$ meets every
non-empty $G_\delta$ subset of $X$.

\begin{corollary} For each compact space $X$ there is a $G_\delta$-dense
  subset $Y$ of $X$ such\label{ctbly}
 that $d(Y) = \pi(X)$.
\end{corollary}

\begin{proof}
Applying  Proposition \ref{delta} we choose  a dense set
 $Y\subset X$ so that no subset of $Y$ of cardinality less than
$\pi(X)$ is dense. Now let $\operatorname{cl}_{\aleph_0}(Y) =
 \bigcup \{ \overline{A} : A\subset Y
 \ \mbox{and}\ \ |A|=\aleph_0\}$. For every infinite cardinal
$\mu \leq |Y|$, each subset of $
\operatorname{cl}_{\aleph_0}(Y) $ of cardinality $\mu$
 is contained in the closure of a subset of $Y$ of cardinality
 $\mu$. Therefore $
\operatorname{cl}_{\aleph_0}(Y) $ has no dense subset of cardinality
 less
 than $\pi(X)$. Since $X$ is compact, $
 \operatorname{cl}_{\aleph_0}(Y) $ is countably
 compact and dense. It is now immediate that $
 \operatorname{cl}_{\aleph_0}(Y) $ is
 $G_\delta$-dense in $X$.    
\end{proof}

\section{Elementary properties of $\delta_k(X)$}

A map $f$ from a space $X$ to a space $Y$ is said to be irreducible if
$f$ is onto and
the images of   proper closed subsets of $X$ are proper subsets of $Y$. 
When $f$ is continuous and irreducible, it follows that any subset of
$X$ that maps to 
a dense subset of $Y$ must itself be dense in $X$.
Each continuous
irreducible map is quasi-open. A map $f$ is quasi-open if the image of
every non-empty open set has non-empty interior.

\begin{proposition}  If $f$ is a continuous irreducible map from a
  space $X$ onto $Y$, then\label{irred}
$\delta(X)=\delta(Y)$,  $\pi(X)=\pi(Y)$,
   and $\delta_k(X)\leq \delta_k(Y)$. 
\end{proposition}

\begin{proof}
  Let $f, X$ and $Y$ be as in the statement.
  If $D\subset Y$ is dense, then the density of $D$ is at most
  the density of 
  $f^{-1}(D)$.  Similarly, if $D$ is a dense subset of $X$,
   then $f(D)$ is dense in $Y$ and so there is a dense
   $E\subset f(D)$ of cardinality at most $\delta(Y)$. 
   Any subset $E'$ of $D$ that maps onto $E$ will be dense in $X$,
    hence the density of $D$ is at most $\delta(Y)$. 
 If $\mathcal B $ is a
$\pi$-base for $X$, then $\{ \intt_Y(f(B)) : B\in \mathcal B\}$ is a
$\pi$-base for $Y$. Conversely, if $\mathcal B$ is  a $\pi$-base for
$Y$,
 then the family $\{ X \setminus f^{-1}(Y\setminus B) : B\in \mathcal
 B\}$ is a $\pi$-base for $X$.  If $D$ is a dense subset of $Y$,
 then $f^{-1}(D)$ is a dense subset of $X$. Fix a family
  $\{ A_\alpha : \alpha < \delta_k(X)\}$ of compact subsets of
 $f^{-1}(D)$ whose union is dense in $X$. It follows that
  $\{ f(A_\alpha) : \alpha< \delta_k(X)\}$ is a family of compact subsets
 of $D$ whose union is dense in $Y$. This proves that $\delta_k(Y)\leq
 \delta_k(X)$. 
\end{proof}

\begin{proposition} The following are equivalent\label{quasi}
 for a continuous map
  from a space $X$  onto $Y$:
  \begin{enumerate}
    \item for each open $U\subset X$, $\intt_Y(f(U))$ is dense in
      $f(U)$,
        \item $f$ is quasi-open,
  \item the pre-image of every dense set is dense
     \end{enumerate}
\end{proposition}

\begin{proof}
The proofs of
 the implications (1) implies (2) and (2) implies (3) are trivial.
 Assume that (3) holds and let $U$ be a non-empty open subset of $X$
 and that $W\cap f(U)$ is not empty for some open $W\subset Y$.
 Let $U_1 = U\cap f^{-1}(W)$. Since $f(U_1)= W\cap f(U)$,
 showing that $f(U_1)$ has non-empty interior will prove
 that (1) holds.  Since $f^{-1}(Y\setminus f(U_1))$ is not
 dense, it follows from (3) 
that $Y\setminus f(U_1)$ is not dense. This implies that $f(U_1)$
has non-empty interior.
\end{proof}

\begin{corollary} If $Y$ is the continuous\label{quasid}
 quasi-open image of a space
  $X$, then $\delta_k(Y)\leq \delta_k(X)$.
\end{corollary}

For a space $X$, the celluarity $c(X)$ is the supremum of the
cardinalities of cellular (pairwise disjoint) families of open subsets
of $X$.
The first inequality in
this next result follows easily from the fact
that a compact subset of the union of a pairwise disjoint family of
open sets will be contained a union of finitely many of them. 
The second is a consequence of the fact that $\delta(X)\leq \pi(X)$.

\begin{proposition} For all $X$,\label{ccc} 
  $c(X) \leq \delta_k(X) \leq \pi(X)$. 
\end{proposition}

Similarly, this next statement is easily proven but will be useful.

\begin{proposition} For all $X$\label{regclosed} and
  all regular closed   $U\subset X$, $c(U)\leq c(X)$,
   $\delta_k(U)\leq \delta_k(X)$, and $\pi(U)\leq \pi(X)$.
\end{proposition}

Say that a compact space $X$  has uniform $\pi$-weight
 if every non-empty
regular closed subset has  $\pi$-weight equal to $\pi(X)$. 

\begin{proposition} Each compact space  $X$\label{uniform},
with   $\delta_k(X) < \pi(X)$, has  a regular closed subset
  $Y$ having uniform
  $\pi$-weight and satisfying that $\delta_k(Y) < \pi(Y)$.
\end{proposition}

\begin{proof} 
  For each non-empty regular-closed  subset $U$ of $X$,
   let $\mathcal R_U$ denote the set of non-empty regular closed
   subsets of $\intt_X(U)$. For each $U\in \mathcal R_X$, 
 let $\delta_U$ be the minimum element of $\{ \pi(W) :  W\in \mathcal
 R_U\}$  and choose any $W_U\in \mathcal R_U$ with $\pi(W_U) =
 \delta_U$. Evidently $\pi(W) = \pi(W_U)$ for all regular closed
 subsets $W$ of $W_U$, that is $W_U$ has uniform $\pi$-weight and each
 element of $\mathcal R_{W_U}$ also has uniform $\pi$-weight. We may
 now choose a maximal family, $\mathcal B$,
of pairwise disjoint members of $\mathcal R_X$ that have uniform
$\pi$-weight. By Proposition \ref{ccc}, $|\mathcal B|\leq \delta_k(X)$.
 If $\mathcal B_U\subset \mathcal R_U$ is a $\pi$-base for
each $U\in \mathcal B$, then $\bigcup \{ \mathcal B_U : U\in \mathcal
B\}$ is a $\pi$-base for $X$. It follows that there is a $Y\in
\mathcal B$ such that $\pi(Y)>\delta_k(X)$.  By Proposition \ref{regclosed},
we have that $\delta_k(Y) \leq \delta_k(X) < \pi(Y)$ as required.
\end{proof}

This next result is a crucial step in the proof of our main result.
In this proof, and others, we will use the notion of elementary
submodels. Most readers will be familiar enough with these notions and
they are surveyed in \cite{elem88}. In particular, for a regular
 cardinal $\theta$, $H(\theta)$ denotes the set of all sets whose
 transitive closure has cardinality less than $\theta$. An elementary
submodel  $M$ (of cardinality less than $\theta$)
 of $H(\theta)$ 
can be thought of as an element of $H(\theta)$ with
 the property that anything true in $H(\theta)$ of finitely many
 elements of $M$ is also true in $M$. The classical
L\"owenheim-Skolem theorem ensures that every infinite element $x$
of $H(\theta)$ is 
a subset of an elementary submodel $M$ of
$H(\theta)$ with $|M| = |x|$. Tarski proved that the union of an
increasing chain of fewer than $\theta$ many  elementary submodels 
of $H(\theta)$
 is an elementary submodel of $H(\theta)$.

\begin{lemma} If $X$ is a compact space with $\delta_k(X) = \kappa <
  \pi(X)$, 
  then there is a quasi-open continuous\label{reflect}
 image $Y$ of $X$ with
 $\delta_k(Y)\leq \kappa < \pi(Y) \leq w(Y) \leq 2^\kappa$. 
\end{lemma}

\begin{proof}
Let $\theta$ be the weight of $X$ and identify $X$ with a closed
subset of $[0,1]^\theta$. Let $M$ be an elementary submodel of
$H(\theta^+)$ such that $X\in M$, $M^\kappa\subset M$, and
$|M|=2^\kappa$. Now let $\pr_M$ denote the projection mapping
from $[0,1]^\theta$ onto $[0,1]^{M\cap \theta}$, and let 
$Y = \pr_M(X)$. Evidently $\pi(Y) \leq w(Y) \leq |M|$. If $\mathcal B$
is a family of at most $\kappa$ many basic open subsets of
$[0,1]^{M\cap \theta}$, then $\mathcal C =
\{ \pr_M^{-1}(B) : B\in \mathcal B\}$ is an element of $M$. Since
$\{ C\cap X: C\in \mathcal C\}$ is not a $\pi$-base for $X$,
it follows by simple elementarity that $\{ B\cap Y : B\in \mathcal
B\}$ is not a $\pi$-base for $Y$.

Now we prove that $\pr_M\restriction X$ is quasi-open.
Since $X$ is regular, it suffices to prove that the image of each
non-empty regular closed set has non-empty interior. Since
the image of such a set is also compact,
 it suffices to prove that the pre-image of
 each dense open subset of $Y$ is a dense subset of $X$.
 Assume that
$D$ is a dense open subset of $Y$.
By Proposition \ref{ccc}, the celluarity of $X$ is at most $\kappa$.
Since $Y$ is a continuous image of $X$, the cellularity of $Y$ is also
at most $\kappa$. 
Choose any family
 $\mathcal B$ of basic open subsets of $[0,1]^{M\cap \theta}$
satisfying that the members of of $\{ B\cap Y : B\in \mathcal B\}$ are
non-empty and pairwise disjoint subsets of $D$ and so that
the union of $\{ B\cap Y : B\in \mathcal B\}$ is a dense subset of
$D$. 
Since $c(Y)\leq \kappa$, the family $\mathcal B$ has cardinality
at most $\kappa$.
Now we have that the family $\mathcal C = 
\{ \pr_M^{-1}(B) : B\in \mathcal B\}$ is a subset of $M$ and,
since $M^\kappa\subset M$, $\mathcal C$ is an element of $M$. 
We observe that for every 
basic open subset $U$ of $[0,1]^{M\cap  \theta}$ that meets
 $Y$, there is a $B\in \mathcal B$ such that $U\cap B\cap Y$ is not
empty.  Equivalently, for every basic open subset $U\in M$
of $[0,1]^\theta$,  there is a $C\in
 \mathcal C$ such that $U\cap C\cap X$ is not empty.
Therefore, by elementarity, we have that
for every basic open subset $U$ of $[0,1]^\theta$,
there is a $C\in \mathcal C$ such that $U\cap C\cap X$ is not
empty. This implies that the union of the family
$\{ C\cap X : C\in \mathcal C\}$ is a dense subset of $X$.
Since $C\subset \pr_M^{-1}(D)$ for all $C\in \mathcal C$, this
completes
the proof.  
\end{proof}

In the investigation of $t_k^*(X)$ in \cite{DowMoore}, it was fruitful
to consider left-separated sequences of compact nowhere dense 
 $G_\delta$ sets. A set is nowhere dense if its closure has empty
 interior.  We
also find that mixing the notion of left-separated sets with coverings
by compact nowhere dense $G_\delta$ sets will be useful. 
A left-separated transfinite
sequence in a space $X$ is a sequence $\{ x_\alpha : \alpha <\mu\}$
indexed by an ordinal $\mu$
that has the property that
 $x_\alpha$ is not in the closure of $\{ x_\beta : \beta < \alpha\}$
for all $\alpha\in \mu$.

\begin{lemma} Assume that $\{ x_\alpha : \alpha < \mu\}$ is a
  left-separated subset of a space $X$, and for each $\alpha< \mu$,
  there is a\label{leftsep}
 compact $G_\delta$ subset $Z_\alpha$ of $X$ such
  that $x_\alpha\in Z_\alpha$ and $x_\beta\notin Z_\alpha$ for all
  $\alpha\neq \beta < \mu$. Then each compact subset  $K$ of
   $\{ x_\alpha : \alpha < \mu\}$ is scattered and countable.
\end{lemma}

\begin{proof}
This lemma is really two
separate results combined into one.  Namely, the first is
 that a compact subset of a left-separated sequence is scattered
  (\cite{JuhGerlits}),  and the second is
that a compact first-countable scattered space is countable. We prove
each statement. First assume that $K$ is a compact subset of the
left-separated sequence $\{ x_\alpha : \alpha < \mu\}$. Since a
 space is scattered if every one of its closed subspaces has an
 isolated point, we may assume that $K$ has no isolated point and
 obtain a contradiction. It is well-known that such a space $K$ has a
 mapping, $f$, onto
 $[0,1]$.  To see this directly, assume that $K$ is
 embedded into $[0,1]^\theta$ where $\theta$ is the weight of
 $K$. Choose any countable elementary submodel $M$ of $H(\theta^+)$
 such
 that $K\in M$.  Now $\pr_M(K)$ is a compact metrizable space which,
by elementarity, has no isolated points. If $\pr_M(K)$ is totally
disconnected, it is a copy of the Cantor space  and otherwise
$\pr_M(K)$ has a non-trivial connected 
component. Any continuous real-valued function on $\pr_M(K)$ that 
is not constant on such a component, will include an interval in its
range. Therefore there is a mapping of
$\pr_M(K)$  onto $[0,1]$. In either case,
 $\pr_M(K)$ maps onto $[0,1]$, and therefore,  so does $K$. Now choose
$\alpha<\mu$ minimal so that $f(K\cap \{x_\beta : \beta < \alpha\})$
has a crowded subset $S$. Each crowded subset of $[0,1]$ has   a countable
crowded subset, and so the cofinality of $\alpha$ is countable. Since
$\{ x_\beta : \beta <\mu \}$ is left-separated, it follows that
 $K\cap \{ x_\beta : \beta < \alpha\}$ maps onto the (compact) closure
$\overline{S}$ 
of $S$. However $\overline{S}$ is uncountable,
while 
$ f(K\cap \{x_\beta : \beta < \gamma\})$  is countable for all
$\gamma<\alpha$. This contradiction proves that $K$ is scattered.

Now with $K$ any compact scattered subset of 
$\{ x_\alpha : \beta<\mu\}$, we have that $K$ is first-countable
since the family $\{ Z_\alpha \cap K : \alpha < \mu\}$ witnesses that
each point of $K$ is a relative $G_\delta$. Since $K$ is compact it is
first-countable. For each $x\in K$, let $\rho_K(x)$ denote the
scattering level of $K$ that contains $x$. For each $x\in K$,
there is a relatively clopen subset $W_x$ of $K$ satisfying that
 $\rho_K(y) < \rho_K(x)$ for all $x\neq y\in W_x$; that is, $W_x$ is a
relatively clopen
 set witnessing that $x$ is an isolated point of
$K\setminus \{ y\in K : \rho_K(y)<\rho_K(x)\}$.  Since $K$ is compact,
 a finite subcollection of $\{ W_x : x\in K\}$ will cover $K$. If $K$
 is uncountable, 
 then $\{ \rho_K(x) : x\in K\ \mbox{and} \ |W_x|>\aleph_0\}$ has a
 minimum element $\delta$. Choose $x\in K$ such that $\rho_K(x)
 =\delta$
 and $|W_x| > \aleph_0$. Note that $W_y$ is countable for all $y\in
 W_x$.  We now have a contradiction because
 $W_x\setminus \{ x\} $ is $\sigma$-compact and yet
  the family $\{ W_y  : y\in W_x\setminus \{x\}\}$ is an open cover
  with no countable subcover. 
\end{proof}

\section{On densely $k$-separable
  compact spaces}

This next result is similar to  an old result of
Malychin (\cite[3.17]{10years}) showing that an
uncountable compact $T_1$-space has cardinality
at least $\mathfrak c$ if  each point  of the space
is  a $G_\delta$.   In particular, if one constructed
 the family $\mathcal Z$ in the statement of
the Lemma to be  upper semi-continuous 
then the conclusion of the Lemma would follow
from   Malychin's result. We instead prove the Lemma
directly since the construction of the family
 $\mathcal Z$ is then more natural and simpler.

\begin{lemma} Let $X$ be a compact ccc space with\label{tree}
no isolated points,  then there is a
  partition $\mathcal  Z$ of $X$ consisting of nowhere dense
compact $G_\delta$'s
  and satisfying that for all non-empty
  regular closed subsets $U$  of $X$, the set
   $\{ Z\in \mathcal Z : U\cap Z\neq \emptyset\}$ has cardinality
  $\mathfrak c$. 
\end{lemma}

\begin{proof}
Let $\theta$ be the weight of
 $X$ and for convenience we regard it as a compact subset of 
$[0,1]^{\theta}$.
Fix a continuous  elementary chain
$\{ M_\alpha : \alpha\in \omega_1\}$ of countable elementary submodels
of $H(\theta^+)$ so that for all $\beta < \alpha$, $X\in M_\beta \in
M_\alpha$. For each countable
partial function $t$ from a subset of $\theta$
into $[0,1]$, we let $[t]$ denote the $G_\delta$ subset of
 $[0,1]^{\theta}$ consisting of all total functions that extend $t$. 
For such countable functions $t$, we also let $[t]_X$ denote the
(possibly empty) set $[t]\cap X$.

For each $\alpha$, set $T_\alpha$ equal to the set of all $t\in
[0,1]^{M_\alpha\cap  \theta}$ such that $[t]_X$ is not empty. Also let
 $T_{\alpha,0} = \{ t\in T_\alpha : \intt_X([t]_X)\neq
\emptyset\}$. The members of the family $\{ [t] : t\in T_\alpha\}$ are
pairwise disjoint. Also, for $\beta < \alpha\in \omega_1$, the sets
 $T_\beta$ and $T_{\beta,0}$ are  elements of $M_\alpha$.
Since $X$ is ccc and is in $M_0$,  $T_{\alpha,0}$ is countable for
all $\alpha\in \omega_1$. 
For this reason  we also have that
$T_{\beta,0}\subset M_\alpha$  for all $\beta < \alpha\in \omega_1$.

The collection 
$T_{\omega_1,0} = \bigcup \{ T_{\alpha,0} : \alpha \in \omega_1\}$ is
a tree when ordered by $\subset$.  We again note that $[t]_X\cap
[t']_X$ is empty if $t$ and $t'$ have no common extension
(incomparable)  in $T_{\omega_1,0}$. Since $X$ is ccc, this tree has no
uncountable antichains.  In fact, we now check that $T_{\omega_1,0}$
also has no uncountable chains. Assume that $t_\beta \subset t_\alpha$
where $\beta<\alpha$, $t_\beta \in T_{\beta,0}$, and $t_\alpha\in
T_{\alpha,0}$. Note that for all $\gamma <\alpha$,
$t_\alpha \restriction M_\gamma$ is also a member
of $T_{\gamma,0}$ since $\intt_X([t_\alpha\restriction M_\gamma]_X)$
contains  $\intt_X([t_\alpha]_X)$.
Since  $X$ has no isolated points, the infinite open set 
$\intt_X([t_\alpha]_X)$ contains a disjoint pair of closed
$G_\delta$-subsets of $X$. In particular, there are
incomparable countable extensions $t_1, t_2$ of $t_\alpha$,
functions from a countable subset of $\theta$ into $[0,1]$,
 such that $[t_1]_X\cup [t_2]_X\subset \intt_X([t_\alpha]_X)$.
Since $t_\beta\in M_{\beta+1}$, 
it then follows by 
elementarity that there are incomparable $t_1,t_2\in T_{\beta+1}$ such
that $[t_1]_X\cup [t_2]_X\subset \intt_X([t_\beta]_X)$. This shows
that $\intt_X([t_\beta]_X) \setminus [t_\alpha]_X$ is not
empty. Therefore, if $\{ t_\alpha : \alpha  < \omega_1\}$ is a
chain in $T_{\omega_1,0}$, we can assume that $t_\alpha\in
T_{\alpha,0}$ for each $\alpha\in \omega_1$, and we have that
the existence of
the family   $\{ \intt_X([t_\alpha]_X) \setminus [t_{\alpha+1}]_X :
\alpha\in \omega\}$ contradicts that $X$ is ccc.

Now we use $T_{\omega_1,0}$ to define a special antichain $T$
in the tree  $T_{\omega_1} = 
\bigcup \{ T_\alpha : \alpha\in \omega_1\}$. 
A node $t$ is in $T$ if there is an $\alpha \in\omega_1$ such
that $t\in T_\alpha\setminus T_{\alpha,0}$ and for all $\beta<\alpha$,
   $t\restriction M_\beta\in T_{\beta,0}$. That is $T$ is the set of
all minimal nodes of $ T_{\omega_1} \setminus T_{\omega_1,0} $. Fix
any $x\in X$ and let $C_x = \{ t\in  T_{\omega_1} : x\in [t]\}$. Since
 $C_x$ is an uncountable chain, it is not  a subset of
$T_{\omega_1,0} $. Therefore there is a minimal $\alpha$ such that
there is a $t_x\in   T_{\alpha}\setminus T_{\alpha,0}$ with $x\in
[t_x]_X$. By the minimality of $\alpha$, $t_x\in T$. This proves
that $\mathcal Z = \{ [t]_X : t\in T\}$ is a partition of $X$.

Now let $U$ be a  non-empty regular closed subset of $X$.
We prove that the
set $\{t\in T: U\cap [t]_X\neq\emptyset\}$ has cardinality $\mathfrak c$. 
 For each non-empty regular closed $W$ of $U$, let
$\mu_W = \sup\{ \alpha \in \omega_1 : (\exists t\in T_{\alpha,0})
 W\cap \intt_X([t]_X)\neq\emptyset \}$. 
 By passing to a regular closed subset of $U$ with a minimum value
 of $\mu_U$, we can assume that
 $\mu_W = \mu_U$ for all regular closed subsets $W$ of $U$. 

 Assume first that there is a $t\in T_{\mu_U,0}$ such that $U\cap
 \intt_X( [t]_X)$ is not empty and let $\alpha = \mu_U+1$.
 Choose any regular closed subset $W$ of $U$
 that is contained in  $ \intt_X([t]_X)$. Since $\mu_W = \mu_U$,
 it follows that $W$ is covered by the family of compact nowhere dense
 sets $\{ W\cap [t']_X : t\subset t'\in T_{\alpha} \}$. By the Baire
 category theorem, this family is uncountable. Also,
  the family $\{ t' \in [0,1]^{M_\alpha\cap \theta } : [t']_X\cap W
  \neq \emptyset \}$ is equal to $\pr_{M_\alpha}(W)$ and so
is an uncountable closed subset of the metric space
   $[0,1]^{M_\alpha\cap \theta }$ and will therefore  have
cardinality $\mathfrak 
  c$.  All but countably many of the $t'\in T_{\alpha}$ with
   $W\cap [t']_X\neq\emptyset$ are in $T$. This completes the proof in
  this case. 

The same proof as in the previous paragraph, with $W=U$,
 shows that if
 $\mu_U=0$ and $U\cap \intt_X([t]_X)$ is empty for all
   $t\in T_{0,0}$, then  $\{ t\in T : U\cap [t]_X\neq\emptyset\}$ has
cardinality $\mathfrak c$. Similarly, 
if $\mu_U$ is a successor
ordinal, $\beta+1$, then we let $W$ be a non-empty regular closed
subset of $U\cap \intt_X([t]_X)$ for any $t\in T_{\beta,0}$ such
that $U\cap \intt_X([t]_X)$ is not empty, and proceed as above.

Now we consider the final case where
$\mu_U$ is a limit ordinal and, 
if $\mu_U<\omega_1$, that  $U\cap \intt_X([t]_X)$ is empty for all
 $t\in T_{\mu_U,0}$. 
Choose any
 $\alpha_0<\mu_U$ and $t_0\in T_{\alpha_0,0}$ such that $W\cap
 \intt_X[t_0]_X$ is not empty. We have started a recursive
 construction of choosing a strictly increasing sequence $\{ \alpha_n
 : n\in \omega\} \subset \mu_U$ and a family $\{ t_s : s\in
 2^{<\omega} = \bigcup_n 2^n\}$ such that, for each $n\in\omega$
 \begin{enumerate}
 \item $U\cap \intt_X([t_s]_X)$ is not empty for each $s\in 2^n$,
   \item $\{ t_s : s\in 2^n\}$ are distinct elements of
     $T_{\alpha_n,0}$
\item $t_{s'} \subset t_s$ for all $s' = s\restriction m$ and $m<n$.
 \end{enumerate}
 Assume that $n\in \omega$ and that we have chosen
 $\alpha_n<\mu_U$ and the sequence
  $\{ t_s : s\in 2^n\}$ so that the above conditions hold. 
 For each $s\in 2^n$, choose a regular closed subset $U_s$ of 
 $ U\cap \intt_X([t_s]_X)$. Since there are no uncountable chains
 in $T_{\omega_1,0}$, there is an $\alpha_s<\omega_1$ such that,
 for some $t_s'\in T_{\alpha_s,0}$, $U_s \cap \intt_X([t_s']_X)$ is
 not empty and $\intt_X(U_s\setminus [t_s'])$ is not empty.
 Since $\mu_W = \mu_U$ for all regular closed subsets of $U$,
 there is also a $t_s''\in T_{\alpha_s,0}$ 
such that the open set 
  $\intt_X(U_s\setminus [t_s'])$ meets $\intt_X([t_s''])$.
Choose any $\alpha_{n+1} < \mu_U$ so that $\alpha_s \leq \alpha_{n+1}$
for all $s\in 2^n$. For each $s\in 2^n$, there are $t_{s^\frown 0}$
and
$t_{s^\frown 1}$ in $T_{\alpha_{n+1},0}$ such that
$t_s'\subset t_{s^\frown 0}$, $t_s''\subset t_{s^\frown 1}$,
$U_s \cap \intt_X([t_{s^\frown 0}]_X)$ and
$U_s \cap \intt_X([t_{s^\frown 1}]_X)$ are not empty. This completes
the recursive construction. Let $\alpha = \bigcup_n \alpha_n$. 
Let $\rho$ be any element of $2^\omega$,
and let $t_\rho = \bigcup\{ t_{\rho\restriction n} : n\in
\omega\}$.
Since $M_\alpha = \bigcup \{ M_{\alpha_n} : n\in \omega\}$,
 $t_\rho\in T_{\alpha}$ for each $\rho\in 2^\omega$. Since $U$ is
compact and $U\cap [t_\rho\restriction M_{\alpha_n}]_X$ is not empty
 for each
$n\in\omega$,  we have that $U\cap [t_\rho]_X$ is not empty for each
$\rho\in 2^\omega$. For each $\rho\in 2^\omega$ such that
$t_\rho\notin T_{\alpha,0}$, it follows that $t_\rho\in T$ since
$t_\rho\restriction \alpha_n \in T_{\alpha_n,0}$ for all $n\in
\omega$.
This completes the proof of the Lemma.
\end{proof}

Now we prove the main theorem.

\begin{theorem} A compact space is densely $k$-separable
if  and only if it has countable $\pi$-weight.
\end{theorem}

\begin{proof}
It follows from Proposition \ref{ccc} that $\delta_k(X) = \aleph_0$ for
all compact spaces of countable $\pi$-weight.  For the other
direction assume
that $X$ is
compact and that $\delta_k(X) = \aleph_0$.
By Proposition \ref{ccc}, $X$ is ccc.
Now we assume
that $\aleph_0 < \pi(X)$ and work towards a contradiction.
  By Lemma
\ref{reflect}, we may replace $X$ by a quasi-open continuous image
and thereby  assume that the weight of $X$ is at most
$\mathfrak c$. By Proposition \ref{uniform}, we can pass to a suitable
regular closed subset of $X$ and thereby assume that $X$ has uniform
$\pi$-weight. Now apply Lemma \ref{tree}, choose a partition $\mathcal
Z$ of $X$ by compact $G_\delta$'s with the property that the set
 $\{ Z\in \mathcal Z : W\cap Z\neq \emptyset\}$ has cardinality
$\mathfrak c$ for each non-empty regular closed subset $W$ of $X$. By
Corollary  \ref{ctbly}, we choose a $G_\delta$-dense subset $Y$ of $X$ so that
$Y$ has no dense subset of cardinality less than $\pi(X)$.
Of course this means 
that $Y$ meets every member of $\mathcal Z$. Fix an enumeration
 $\{ U_\alpha : \alpha < \pi(X)\}$ of a family of regular closed sets
whose interiors form a $\pi$-base for $X$. Let $y_0$ be any element
of $Y\cap \intt_X(U_0)$ and let $Z_0\in \mathcal Z$ be such
that $y_0\in Z_0$. Assume, by induction that $\alpha<\pi(X)$ and we
have chosen a  sequence $\{ y_\beta : \beta < \alpha\}\subset Y$
and a sequence $\{ Z_\beta : \beta < \alpha\}\subset \mathcal Z$ so
that for all $\beta < \alpha$
\begin{enumerate}
  \item  $y_\beta$ is not in the closure of $\{ y_\xi : \xi <\beta\}$,
  \item $Z_\beta\notin \{ Z_\xi : \xi<\beta\}$,
    \item $\{ y_\xi : \xi\leq \beta\}$  meets $ \intt_X(U_\beta)$.
\end{enumerate}
We choose $y_\alpha$ and $Z_\alpha$ as follows. Since
$\{ y_\beta : \beta < \alpha\}$ is not dense, we may choose
the minimal  $\gamma_\alpha<\pi(X)$ such that
$\intt_X(U_{\gamma_\alpha})$ is disjoint from 
$\{ y_\beta : \beta < \alpha\}$. 
Choose $\delta_\alpha$ so
that $U_{\delta_\alpha}$ is a subset of
 $\intt_X(U_{\gamma_\alpha})\setminus \overline{\{y_\beta: \beta <
  \alpha\}}$. We note, by induction assumption (3),  that
 if $\{y_\beta : \beta < \alpha\}$ is
disjoint from $\intt_X(U_\alpha)$, then $\gamma_\alpha = \alpha$.
 Choose any
$Z_\alpha\in \mathcal Z\setminus \{ Z_\beta : \beta < \alpha\}$ so
that $Z_\alpha\cap U_{\delta_\alpha}$ is not empty. Since $Y$ is
$G_\delta$-dense, we can choose $y_\alpha\in Z_\alpha\cap
\intt_X(U_{\gamma_\alpha})$. This completes the inductive construction
of the left-separated sequence $\{ y_\alpha : \alpha < \pi(X)\}$
together with the sequence $\{ Z_\alpha : \alpha < \pi(X)\}\subset
\mathcal Z$. By Lemma \ref{leftsep}, each compact subset of
 $\{ y_\alpha : \alpha < \pi(X)\}$ is countable. 
This proves that $\{ y_\alpha : \alpha < \pi(X)\}$ does not have a
$\sigma$-compact subset that is a dense subset of $Y$. However, we
have ensured that $\{ y_\beta : \beta \leq \alpha\}\cap
\intt_X(U_\alpha)$ is not empty for all $\alpha<\pi(X)$,
and so $\{ y_\alpha : \alpha <\pi(X)\}$ is a dense subset of $X$.
This contradicts that $\delta_k(X) = \aleph_0$.
\end{proof}

\end{document}